\documentclass{article}
\usepackage{amsmath,amsfonts,amsthm,graphics,color}
\usepackage{natbib}
\bibpunct[, ]{[}{]}{;}{n}{,}{,}

\newtheorem{theorem}{Theorem}[section]
\newtheorem{lemma}[theorem]{Lemma}

\newtheorem{corollary}[theorem]{Corollary}
\theoremstyle{definition}

\newtheorem{remark}[theorem]{Remark}

\newtheorem*{ack}{Acknowledgement}

\newcounter{thmenumerate}

\newcommand\Tk{T_{\ka}}

\newcommand\RR{{\mathbb R}}
\newcommand\eps{\varepsilon}
\newcommand\sss{{\mathcal S}}
\newcommand\dd{\,d}
\newcommand\norm[1]{\ensuremath{\|#1\|}}

\newcommand\upto{\nearrow}

\newcommand{\refT}[1]{Theorem~\ref{#1}}

\newcommand{\refC}[1]{Corollary~\ref{#1}}
\newcommand{\refL}[1]{Lemma~\ref{#1}}

\newcommand\E{{\mathop{\mathbb E{}}\nolimits}}

\newcommand\ka{\kappa}

\newcommand\bpk{{\mathfrak X}_{\ka}}

\newcommand\pto{\overset{\mathrm{p}}{\to}}

\newcommand\aut{\operatorname{aut}}
\newcommand\cc{{\mathrm{c}}}
\newcommand\dcut{{\delta_{\square}}}

\newcommand\set[1]{\ensuremath{\{#1\}}}

\newcommand\bigpar[1]{\bigl(#1\bigr)}

\newcommand\Bigabs[1]{\Bigl|#1\Bigr|}

\newcommand\lrabs[1]{\left|#1\right|}

\newcommand\oi{[0,1]}

\newcommand\ntoo{\ensuremath{{n\to\infty}}}

\newcommand\cC{\mathcal C}

\newcommand\cW{\mathcal W}

\newcommand\qw{^{-1}}

\newcommand\op{o_{\mathrm{p}}}

\newcommand\gl{\lambda}

\newcommand\taun{\tau_n}

\renewcommand{\=}{:=}

\newcommand\ooo{[0,\infty)}
\newcommand\rplus{[0,\infty)}

\newcommand\glw{\gl_W}
\newcommand{\txp}{t_{\mathrm{isol}}^{+}}
\newcommand{\txx}{t_{\mathrm{isol}}^{\times}}
\newcommand{\tx}{t_{\mathrm{isol}}}
\newcommand{\ty}{t_{0}}

\newcommand\cn[1]{\norm{#1}\cut}
\newcommand\cnm[2]{\norm{#1}_{\square,#2}}

\newcommand\cut{_{\square}}
\newcommand\sij{_{ij}}
\newcommand\ab{^{(a,b)}}
\newcommand\wab{W\ab}
\newcommand{\cw}{\cW}

\newcommand{\cws}{\cW_{\mathrm{sym}}}

\newcommand\on[1]{\norm{#1}_{L^1}}

\newcommand\sn[1]{\norm{#1}_{\infty}}

\newcommand\ta{{(\tau)}}

\newcommand\hka{\widehat\kappa}
\newcommand\hhka{\widehat\kappa'}
\newcommand\tka{\widetilde\kappa}
\newcommand\hmu{\widehat\mu}

\newcommand\tG{{\widetilde G}}
\newcommand\tA{{\widetilde A}}
\newcommand\iiss{\int_{\sss^2}}

\newcommand\tf{{\tilde f}}
\newcommand\tg{{\tilde g}}

\begin{document}
\title{Duality in inhomogeneous random graphs, and the cut metric}

\author{Svante Janson%
\thanks{Department of Mathematics, Uppsala University,
 PO Box 480, SE-751 06 Uppsala, Sweden}
\and Oliver Riordan%
\thanks{Mathematical Institute, University of Oxford, 24--29 St Giles', Oxford OX1 3LB, UK}}
\date{May 1, 2009} 


\maketitle

\begin{abstract}

The classical random graph model $G(n,\lambda/n)$
satisfies a `duality principle', in that removing the giant component
from a supercritical instance of the model leaves (essentially) a
subcritical instance. Such principles have been proved for various
models; they are useful since it is often much easier to study
the subcritical model than to directly study small components
in the supercritical model. Here we prove a duality principle of this
type for a very general class of random graphs with independence
between the edges, defined by convergence of the matrices
of edge probabilities in the cut metric.

\end{abstract}

\section{Introduction and results}

Throughout, a matrix denoted $A_n$ is assumed to be symmetric, $n$-by-$n$, and
to have non-negative entries.
Given such a matrix $A_n=(a_{ij})$,
let $G(A_n)$ denote the random graph on $[n]=\{1,2,\ldots,n\}$ in which edges
are present independently and the probability that $ij$ is an edge is $\min\{a_{ij}/n,1\}$.
If $A_n$ is itself random, then $G(A_n)$ denotes the random graph whose conditional
distribution, given $A_n$, is as above.
As shown by Bollob\'as, Janson and Riordan~\cite{cutsub}, if the matrices
$A_n$ converge (in probability) in a certain sense defined below,
then the random graph `model' $G(A_n)$ may be seen as a generalization
of many earlier inhomogeneous models, such as that introduced in~\cite{kernels}.
Furthermore, results for $G(A_n)$ generalize corresponding 
results for percolation on sequences of dense finite graphs
of the type proved by Bollob\'as, Borgs, Chayes and Riordan~\cite{QRperc}. 

It is well known that in the classical random graph $G(n,p)$, $p=\gl/n$,
the small components of the supercritical
graph behave like a subcritical instance of the same model;
this fact was first exploited by Bollob\'as~\cite{BB84}.
It was also used by {\L}uczak~\cite{Luczak},
who stated it explicitly
as the `symmetry rule'; see also~\cite{JLR}.
It is also sometimes known as a (discrete) `duality principle'; see, for example,
Alon and Spencer~\cite{AS}.
Corresponding results have been proved for several other models,
for example by \citet{MRsize} for the configuration model of
Bollob\'as \cite{B1},
and by \citet{kernels} for their inhomogeneous model.
Our aim in this note is to prove such a result for the very
general model $G(A_n)$ described above.

First we need a few definitions, mainly from~\cite{cutsub}, although many of the
important concepts are from earlier papers.
Let $(\sss,\mu)$ be a measure space with $0<\mu(\sss)<\infty$. 
Almost all the time, $\mu$ will be a probability measure; in fact, most of the time we shall
take $\sss$ to be $[0,1]$ (or $(0,1]$) with $\mu$ Lebesgue measure.
A {\em kernel} on $\sss$ is
an integrable, symmetric function $\ka:\sss^2\to \rplus$. 
Adapting a definition of Frieze and Kannan~\cite{FKquick}, for $W\in L^1(\sss^2)$ we define
the {\em cut norm} $\cn{W}$ of $W$ by
\begin{equation}\label{cntdef}
 \cn W \= \sup_{\sn{f},\sn{g}\le1}
  \Bigabs{\int_{\sss^2} f(x)W(x,y)g(y)\dd\mu(x)\dd\mu(y)}.
\end{equation}
(This is equivalent within a factor $4$ to the variant where $f$ and $g$ are $0/1$-valued functions.)
A {\em rearrangement} of the kernel $\ka$ is any kernel $\ka^\ta$ defined by
\begin{equation}\label{kata} 
 \ka^\ta(x,y) = \ka(\tau(x),\tau(y)),
\end{equation}
where $\tau:\sss\to\sss$ is a measure-preserving bijection.
Given two kernels $\ka$, $\ka'$ on $[0,1]$,
the {\em cut metric} of
Borgs, Chayes, Lov\'asz, S\'os and Vesztergombi~\cite{BCLSV:1}
may be defined by 
\begin{equation}\label{cutdef}
 \dcut(\ka_1,\ka_2) = \inf_\tau \cn{\ka_1-\ka_2^\ta},
\end{equation}
where the infimum is over all rearrangements of $\ka_2$. (Of course, it makes no
difference if we rearrange $\ka_1$ instead, or both $\ka_1$ and $\ka_2$.)

Probabilistically, it is more natural to define $\dcut$ via couplings, as
discussed in~\cite{BCLSV:1}; see also~\cite{BRmetrics}.
Given two measure spaces $(\sss_1,\mu_1)$, $(\sss_2,\mu_2)$ with $0<\mu_1(\sss_1)=\mu_2(\sss_2)<\infty$,
a {\em coupling} of these spaces is simply a measure space $(\sss,\mu)$ together with measure preserving
maps $\sigma_i:\sss\to\sss_i$, $i=1,2$. Given kernels $\ka_i$ on $\sss_i$,
the corresponding {\em pull-backs} $\ka_i^{(\sigma_i)}$ are the kernels on $(\sss,\mu)$ defined by
\[
 \ka_i^{(\sigma_i)}(x,y) = \ka_i(\sigma_i(x),\sigma_i(y)),
\]
and the cut metric may be defined by
\[
 \dcut(\ka_1,\ka_2) = \inf \cn{\ka_1^{(\sigma_1)} - \ka_2^{(\sigma_2)}},
\]
where the infimum is taken over all couplings.
It is not obvious that this definition agrees with \eqref{cutdef} for kernels on $[0,1]$, but this
turns out to be the case, as shown in~\cite{BCLSV:1}.

Although the coupling definition is perhaps more natural (and is forced
on us if we consider probability spaces with atoms), the rearrangement definition
seems intuitively simpler, and is often notationally simpler. Where possible, we shall
work with rearrangements rather than couplings. However, we shall still need to consider
kernels on different spaces. In this setting
a {\em rearrangement} of a kernel $\ka$ on $(\sss_1,\mu_1)$ is any kernel $\ka^\ta$ 
on $(\sss_2,\mu_2)$, where
$\tau:\sss_2\to\sss_1$
is a measure preserving bijection
and $\ka^\ta$ is defined by \eqref{kata} as before. In fact, for technical reasons
it is convenient to allow $\tau$ to be a measure-preserving bijection between 
$\sss_2\setminus N_2$ and $\sss_1\setminus N_1$, where the $N_i$ are
null sets: $\mu_i(N_i)=0$.

Given a symmetric $n$-by-$n$ matrix $A_n$, there is a piecewise constant 
kernel $\ka_{A_n}$ on $[0,1]$ naturally associated to $A_n$,
taking the value $a_{ij}$ on the square $((i-1)/n,i/n]\times ((j-1)/n,j/n]$.
(When working with couplings, one can simply view $A_n$ itself as a kernel
on a finite space with $n$ points.)
We often identify $A_n$ and $\ka_{A_n}$, writing, for example, $\dcut(A_n,\ka)$
for $\dcut(\ka_{A_n},\ka)$. Throughout we consider the following random graph
`model': we have a kernel $\ka$ on $[0,1]$ and a sequence $A_n$ of (deterministic
or random) matrices with $\dcut(A_n,\ka)\pto 0$, and study $G_n=G(A_n)$.
We shall show that deleting the giant component from such a graph $G_n$,
when it exists, leaves another instance of the same model. To make sense
of this requires some further definitions.

Let $\Tk$ denote the integral operator associated to $\ka$, defined
by $(\Tk f)(x)=\int_\sss \ka(x,y)f(y)\dd\mu(y)$.

Given a kernel $\ka$ on a type space $(\sss,\mu)$,
where $\mu$ is a probability measure, let $\bpk$ be the
Poisson Galton--Watson branching process
naturally associated to $\ka$: we start with a single particle whose type is distributed
according to $\mu$, particles have children independently of each other and
of the history, and the types of the children of a particle of type
$x$ form a Poisson process on $\sss$ with intensity $\ka(x,y)\dd\mu(y)$.
We write $\bpk(x)$ for the same process started with a single particle of type $x$.

As in~\cite{kernels}, let $\rho(\ka)$ denote the survival probability of $\bpk$
and $\rho(\ka;x)$ that of $\bpk(x)$.
Also, let $\rho_k(\ka;x)$ and $\rho_k(\ka)$ denote respectively the probabilities
that $\bpk(x)$ or $\bpk$ consists of exactly $k$ particles in total.

We now turn to the `dual' of a kernel $\ka$ on
a probability space $(\sss,\mu)$, 
giving two versions with slightly different normalization.
First, let $\hka$ be the kernel that is equal to $\ka$ as a function, but defined
on the space $(\sss,\hmu)$, where $\hmu$ is the measure
defined by 
\begin{equation}\label{hmudef}
 \dd\hmu(x)=(1-\rho(\ka;x))\dd\mu(x).
\end{equation}
Note that $\hmu(\sss)=1-\rho(\ka)$.
Second, to return to a probability space, let $\hmu'$ be the normalized
measure $\hmu/(1-\rho(\ka))$, and let $\hhka$ be the kernel on $(\sss,\hmu')$
equal to $\ka$ as a function. Finally,
let $\tka=(1-\rho(\ka))\hhka$ be the kernel on $(\sss,\hmu')$
given by $\tka(x,y)=(1-\rho(\ka))\ka(x,y)$.
The kernels $\hka$ and $\tka$ are equivalent in a certain natural sense;
for example, the operators $T_{\hka}$ and $T_{\tka}$ coincide.

Finally, the kernel $\ka$ is {\em reducible} if there is some $A\subset \sss$ with $0<\mu(A)<1$
such that $\ka$ is zero a.e. on $A\times A^\cc$, and {\em irreducible} otherwise.

We write $\cC_i(G)$ for the $i$th largest component of a graph;
for definiteness, if there is a tie, we order components 
of equal sizes according to any fixed ordering on the subsets of $[n]$.
Let $\tG$ denote the graph formed from $G$ by deleting $\cC_1(G)$.
Recall from~\cite{cutsub}
that if $\dcut(A_n,\ka)\pto 0$ and $\ka$ is irreducible, then 
\begin{equation}\label{gc}
 |\cC_1(G_n)|/n \pto \rho(\ka)
\end{equation}
and
\begin{equation}\label{2nd}
 |\cC_2(G_n)|/n \pto 0,
\end{equation}
where $G_n=G(A_n)$. Recall also from~\cite{kernels} that $\rho(\ka)>0$ if and only
if $\norm{\Tk}>1$.

Given a (symmetric, $n$-by-$n$, non-negative, as always)
matrix $A_n$, let $\tA_n$ denote the
random $|\tG|$-by-$|\tG|$ sub-matrix of $A_n$ corresponding to $\tG$, where $G=G(A_n)$.
More precisely, $\tA_n$ may be defined ordering the vertices of $\tG$
arbitrarily, and setting $\tA_{ij}=a_{vw}$
where $v$ and $w$ are the $i$th and $j$th vertices of $\tG$.

Our aim in this paper is to prove the following `duality' result.

\begin{theorem}\label{th1}
Let $(A_n)$ be a
(random or deterministic) 
sequence of symmetric, non-negative matrices
with $\dcut(A_n,\ka)\pto 0$ for some irreducible kernel $\ka$
on $[0,1]$.
Then $\dcut(\tA_n,\hhka)\pto 0$.
\end{theorem}

The main significance is the following consequence.

\begin{theorem}\label{tc1}
Let $(A_n)$ be a sequence of symmetric, non-negative matrices
with $\dcut(A_n,\ka)\pto 0$ for some irreducible kernel $\ka$
with $\rho(\ka)>0$, 
and let $G_n=G(A_n)$.
Then there is a random sequence $(B_n)$ of matrices
such that $\tG_n$ and $G(B_n)$ may be coupled to agree whp,
with $B_n$ $m(n)$-by-$m(n)$, $m(n)/n\pto 1-\rho(\ka)$, and
$\dcut(B_n,\tka)\pto 0$, 
\end{theorem}

\begin{proof}
Conditioning on the $A_n$,
we may assume without loss of generality that the $A_n$ are deterministic, with
$\dcut(A_n,\ka)\to 0$.

The result is essentially immediate from \refT{th1} and the uniqueness of the giant 
component in $G_n$. Indeed, we simply take $B_n=\frac{m(n)}{n}\tA_n$.
Note that $m(n)=n-|\cC_1(G_n)|$ satisfies $m(n)/n\pto 1-\rho(\ka)$
by \eqref{gc}. Since $\dcut(\tA_n,\hhka)\pto 0$ by \refT{th1}, this implies
$\dcut(B_n,\tka)\pto 0$.
Note that $B_n$ depends on $G_n$, but only via the vertex set of $\cC_1(G_n)$.
Conditioning on this vertex set, we see that the distribution
of $\tG_n$ is exactly that of $G(B_n)$ conditioned on containing
no component larger than $\cC_1(G_n)$ (or of the same size but earlier
in our fixed order). However, the unconditional
probability of $G(B_n)$ containing such a component tends to 0,
as otherwise $G_n$ would have positive probability of containing two
components 
of order $\Theta(n)$, contradicting~\eqref{2nd}. For full details of a related 
argument see~\cite[page 79]{kernels}.
\end{proof}

In turn, \refT{tc1} implies, for example, that the number of edges in
the giant component of $G(A_n)$ is `what one would expect',
i.e., that Theorem 3.5 of~\cite{kernels} extends to this more general setting.

\begin{corollary}\label{c2}
Let $\ka$ be an irreducible kernel, and let $G_n=G(A_n)$, where $\dcut(A_n,\ka)\pto 0$.
Then
\[
 \frac{1}{n} e(\cC_1(G_n))\pto \zeta(\ka),
\]
where
\[
  \zeta(\ka) \=
 \frac12
 \int_{\sss^2}\ka(x,y)\bigpar{\rho(\ka;x)+\rho(\ka;y)-\rho(\ka;x)\rho(\ka;y)}
 \dd\mu(x)\dd\mu(y).
\]
\end{corollary}
\begin{proof}
As usual, we condition on the $A_n$ assuming that $\dcut(A_n,\ka)\to 0$.
Next we eliminate `large' entries (in particular those exceeding $n$),
as well as any diagonal entries.

If $\dcut(A_n,\ka)\to 0$, then, as shown in~\cite[Lemma 2.1]{cutsub},
there is some $M(n)$ with $M(n)/n\to 0$ such that the sum of the entries
of $A_n$ exceeding $M(n)$ is $o(n^2)$. Define $A_n'$ by setting
all such entries, and all diagonal entries, to 0. Noting
that the sum of the diagonal entries of $A_n$ not exceeding 
$M(n)$ is at most $nM(n)=o(n^2)$, we have
$\dcut(A_n,A_n')\le \on{\ka_{A_n}-\ka_{A_n'}}=o(1)$, and in the natural coupling $G(A_n)$
and $G(A_n')$ agree in all but $\op(n)$ edges.
The expected number of edges
in $G(A_n')$ is simply $n/2$ times $\int \ka_{A_n'}$.
Since the actual number is a sum of independent indicator variables, its
variance is at most its mean, and hence $O(n)$. Thus
\[
 n^{-1} e(G(A_n)) = n^{-1} e(G(A_n')) + \op(1) = \frac{1}{2} \iiss \ka_{A_n'} +\op(1) = 
 \frac{1}{2}\iiss \ka +\op(1).
\]
Applying this result to $\tG_n$, which agrees
whp with $G(B_n)$, we see that
\begin{align*}
 \frac{1}{n} e(\tG_n) &= \frac{|\tG_n|}{2n} \iiss \tka(x,y)\dd\hmu'(x)\dd\hmu'(y) +\op(1) \\
 &= (1-\rho(\ka))\frac{1}{2}\iiss (1-\rho(\ka))\hhka(x,y)\dd\hmu'(x)\dd\hmu'(y) +\op(1) \\
 &= \frac{1}{2} \iiss \hka(x,y)\dd\hmu(x)\dd\hmu(y) +\op(1) \\
 &= \frac{1}{2} \iiss \ka(x,y)(1-\rho(\ka;x))(1-\rho(\ka;y)) \dd\mu(x)\dd\mu(y) +\op(1).
\end{align*}
Subtracting from $e(G_n)$ gives the result.
\end{proof}

\refT{tc1} has more substantial applications, allowing other quantities associated
to the small components of a suitable random graph $G_n$ to be studied
in a simple way.
For one example, concerning susceptibility, see~\cite{suscep}.
For another, consider Theorem 3 in~\cite{QRperc}. Translated to the present 
notation, this result concerns the graphs $G_n=G(A_n)$, where
the matrices $A_n$ have uniformly bounded entries and $\dcut(A_n,\ka)\to 0$.
It makes two statements: (a) when $\norm{\Tk}<1$ then
$|\cC_1(G_n)|\le B\log n$ holds whp for some constant $B$
(depending on $\ka$ and the bound on the entries of the $A_n$)
and (b) when $\norm{\Tk}>1$ and $\ka$ is irreducible,
then $|\cC_2(G_n)|\le B'\log n$ whp for some $B'$. The proof of part (a)
in~\cite{suscep} is very simple, that of part (b) rather lengthy.
Using~\refT{tc1} it is easy to deduce part (b) from part (a);
one only needs the simple fact that in this setting,
since $\ka$ is bounded and hence $\Tk$ is Hilbert--Schmidt, the dual
kernel is strictly subcritical; see~\cite[Theorem 6.7]{kernels}.

\begin{remark}
Theorems \ref{th1} and \ref{tc1} extend {\em mutatis mutandis} to the
graphs $G(H_n)$ studied in~\cite[Section 3]{cutsub}, which
may be seen as the simple graphs underlying random
(non-uniform) hypergraphs whose `hypermatrices' of edge probabilities
converge in a suitable sense to a `hyperkernel', i.e.,
a sequence of symmetric functions $\ka_r$ on $\sss^r$, $r=2,3,\ldots$.
Since the changes needed are very simple, but complicate the notation, 
we do not give the details. Note that for the analogue of \refC{c2},
one needs an additional condition, called `edge integrability'
in~\cite[Remark 3.5]{cutsub}, as well as convergence in the corresponding
version of the cut metric.
\end{remark}

\section{Proofs}

The main idea is to prove an analogue 
of~\cite[Theorem 9.10]{kernels}.
The statement, \refT{tlf} below, is a little awkward, as we are trying to formulate a result about
the `type' of a vertex in a setting where individual vertices don't really have types.

We start with a much simpler statement concerning branching processes.
As in~\cite{cutsub} we write $\cw$ for the set of all integrable non-negative
functions $W:\sss\times\sss\to\ooo$, and
$\cws$ for the subset of symmetric functions, i.e., kernels.
For $W\in \cw$, we write
$\glw$ and $\glw'$ for the marginals of $W$ with respect to the first
and second 
variables:
\[
 \glw(x) \= \int W(x,y)\dd\mu(y), \qquad \glw'(y) \= \int W(x,y)\dd\mu(x).
\]
Of course, for $W\in\cws$ we have $\glw=\glw'$.

Given a finite graph $F$ with vertex set \set{1,\dots,r},
integrable functions $f_i:\sss\to\RR$, and $W\in \cws$, let
\begin{equation}\label{txf}
  \txx(F,(f_i),W)
\=
 \int_{\sss^r} \prod_{ij\in E(F)} W(x_i,x_j) 
 \prod_{k=1}^{r} f_k(x_k) e^{-\glw(x_k)}
 \dd \mu(x_1)\dotsm \dd \mu(x_r).
\end{equation}
Note that this differs from the quantity $\tx(F,W)$ considered
in~\cite{cutsub} 
by the inclusion of the factors $f_k(x_k)$, $k=1,\ldots,r$.

\begin{lemma}\label{l_FWx}
Let $F$ be a tree and $f_1,\ldots,f_{|F|}$ bounded functions on $\sss$.
Then $W\mapsto \txx(F,(f_k),W)$ is a bounded map on $\cws$ that is
Lipschitz continuous in the cut norm. More specifically, there exists a
constant $C$ (depending on $F$ only) such that 
$|\txx(F,(f_k),W)|\le C\prod_k\sn{f_k}$ for all
$W\in\cws$, and $|\txx(F,(f_k),W)-\txx(F,(f_k),W')|\le C \cn{W-W'}\prod_k\sn{f_k}$ for all
$W,W'\in\cws$.
\end{lemma}
\begin{proof}
The proof is a simple extension of \cite[Theorem 2.3]{cutsub}, so we
only outline 
the differences.

Firstly, writing each $f_k$ as the sum of its positive and negative parts,
we may assume without loss of generality that $f_k\ge 0$ for each $k$.
Also, we may rescale so that $\sn{f_k}=1$ for all $k$.

Given a tree $F$ with $r$ vertices in which each edge has an arbitrary direction,
and for every edge $ij\in F$ a (not necessarily symmetric)
kernel $W\sij\in \cw$, set
\begin{equation}\label{ty}
 \ty\bigpar{F,(W\sij)_{ij\in E(F)}}
\=
\int_{\sss^r} \prod_{ij\in E(F)} W_{ij}(x_i,x_j) 
\dd \mu(x_1)\dotsm \dd \mu(x_r).
\end{equation}
Note that we have omitted both the exponential factors and the factors $f_k(x_k)$
from \eqref{txf}. As in~\cite{cutsub}, given $W\in \cw$ let
\begin{equation}\label{Wabdef}
  \wab(x,y)\=e^{-a\glw(x)} W(x,y) e^{-b\glw'(y)}.
\end{equation}
Also, let
\[
 W_{(ij)} \= f_i(x)^{1/d_i} W(x,y) f_j(y)^{1/d_j},
\]
where $d_i$ is the degree of vertex $i$ in $F$.
It is shown in~\cite[Lemma 2.4]{cutsub} that the map $W\mapsto \wab$ is
Lipschitz continuous with respect to the cut norm, with the constant
independent of $a$ and $b$.
Since $\sn{f_i},\sn{f_j}\le 1$, the linear map $W\mapsto W_{(ij)}$
cannot increase the cut norm, so it and the composition
$W\mapsto \wab_{(ij)}$ are Lipschitz continuous.
Noting that
\[
  \txx(F,(f_k),W)
 =
 \ty\bigpar{F,(W^{(1/d_i,1/d_j)}_{(ij)})\sij},
\]
and that the marginals of $W^{(a,b)}_{(ij)}$ are at most those of $\wab$ and
are hence bounded by constants depending only on $a$ and $b$, 
the rest of the proof of~\cite[Theorem 2.3]{cutsub} goes through unchanged.
\end{proof}
 
\refL{l_FWx} corresponds roughly to counting tree components
of a given size in a certain random graph
by a weight which is a product of the weights of their vertices. In fact, we wish to count
{\em vertices} in such trees by a certain weight, i.e., to count trees by a weight that is 
the sum of the weights of their vertices.

Given a finite graph $F$ with vertex set \set{1,\dots,r},
an integrable function $f:\sss\to\RR$, and $W\in \cws$, let
\begin{equation}\label{txfp}
  \txp(F,f,W)
\=
 \int_{\sss^r} \sum_{k=1}^r f(x_k) \prod_{ij\in E(F)} W(x_i,x_j) 
 \prod_{k=1}^{r} e^{-\glw(x_k)}
 \dd \mu(x_1)\dotsm \dd \mu(x_r).
\end{equation}

\begin{lemma}\label{l_FWp}
Let $F$ be a tree and $f$ a bounded function on $\sss$.
Then $W\mapsto \txp(F,f,W)$ is a bounded map on $\cws$ that is
Lipschitz continuous in the cut norm. More specifically, there exists a
constant $C$ (depending on $F$ only) such that 
$|\txp(F,f,W)|\le C\sn{f}$ for all 
$W\in\cws$, and $|\txp(F,f,W)-\txp(F,f,W')|\le C \cn{W-W'}\sn{f}$ for all
$W,W'\in\cws$.
\end{lemma}
\begin{proof}
Write $\txp(F,f,W)$ as a sum of $|F|$ terms $\txx(F,(f_k),W)$; in each, one of the
$f_k$ is equal to $f$, and the others are the constant function $1$.
\end{proof}

Although we shall not use this, let us note a corollary.
\begin{corollary}
Let $\ka_n$ be a sequence of kernels with $\cn{\ka_n-\ka}\to 0$.
Then for each fixed $k$ we have $\on{\rho_k(\ka_n;\cdot)-\rho_k(\ka;\cdot)}\to 0$,
and $\on{\rho(\ka_n;\cdot)-\rho(\ka;\cdot)}\to 0$.
\end{corollary}
\begin{proof}
It is not hard to check that for any kernel $\ka'$ and any bounded $f$ we have
\[
 \int_\sss \rho_k(\ka';x) f(x) \dd\mu(x)  =  \sum_T \frac{1}{\aut(T)} \txp(T,f,\ka'),
\]
where the sum is over all isomorphism classes of trees on $k$ vertices.
(This generalizes (43) in~\cite{cutsub}; it is perhaps most easily seen
by considering a finite random graph associated to $\ka'$.)
\refL{l_FWp} thus gives
\[
 \lrabs{ \int_\sss (\rho_k(\ka_n;x)-\rho_k(\ka;x))f(x)\dd\mu(x) } \le C'\sn{f}\cn{\ka_n-\ka}
\]
for some constant $C'$. Taking $f(x)$ to be the sign of
$(\rho_k(\ka_n;x)-\rho_k(\ka;x))$, 
the first statement follows.

Turning to the second statement, first note that, summing over $k'\le k$, 
we have
\begin{equation}\label{sk}
 \on{\rho_{\le k}(\ka_n;\cdot)-\rho_{\le k}(\ka;\cdot)} \to 0
\end{equation}
for any fixed $k$. Let
\[
 \Delta_k(\ka')\=\rho(\ka')-\rho_{\le k}(\ka') = \on{\rho(\ka';x)-\rho_{\le k}(\ka';x)}.
\]
From \eqref{sk} and the triangle inequality, for any $k$ we have
\[
 \limsup_\ntoo \on{\rho(\ka_n;\cdot)-\rho(\ka;\cdot)} \le \Delta_k(\ka)+\limsup_\ntoo \Delta_k(\ka_n).
\]

With $k$ fixed, from \eqref{sk} we have $\rho_{\le k}(\ka_n)\to \rho_{\le k}(\ka)$,
so $\limsup_\ntoo \Delta_k(\ka_n)\le \Delta_k(\ka)+\limsup_\ntoo|\rho(\ka_n)-\rho(\ka)|$.
Theorem 1.9 from \cite{cutsub} tells us that $\rho(\ka_n)\to\rho(\ka)$, so this gives
$\limsup_\ntoo \Delta_k(\ka_n)\le \Delta_k(\ka)$, and hence
\[
 \limsup_\ntoo \on{\rho(\ka_n;\cdot)-\rho(\ka;\cdot)} \le  2\Delta_k(\ka)
\]
for any $k$. Letting $k\to \infty$, noting that $\rho_{\le k}(\ka)\upto \rho(\ka)$,
we have $\Delta_k(\ka)\to 0$, and the result follows.
\end{proof}

We now turn to the random graph equivalent of \refL{l_FWp}, again
using methods from~\cite{cutsub}. 
In the sequel we will for convenience take $(\sss,\mu)$ to be
$\oi$ with Lebesgue measure, but we will continue to write
$\sss$ and $\mu$ to emphasize that the results
easily extend to general spaces $\sss$. 
Let $(A_n)$ be a sequence of matrices and $\ka$ a kernel 
on $\sss=\oi$
with $\dcut(A_n,\ka)\to 0$,
and let $\ka_n'=\ka_{A_n}^{(\tau_n)}$ be a rearrangement of
$\ka_{A_n}$ chosen so that $\cn{\ka_n'-\ka}\to 0$. 
We write
\begin{equation}\label{Sv}
 S_v=S_{v,n}=\taun\qw((v-1)/n,v/n])
\end{equation}
for the subset of $\sss$
  corresponding to the vertex $v$ under 
this rearrangement.
Given a sequence $f_n$ of integrable functions on $\sss$, for $v\in V(G_n)$ set
\begin{equation}\label{fnv}
 f_n(v) = n \int_{S_{v,n}} f_n(x)\dd\mu(x),
\end{equation}
so $f_n(v)$ is the average of $f_n$ over $S_v$.
Note that if $\ka$ is finite type, $f_n$ depends only on the type, and
the rearrangement $\taun\qw$ maps each vertex into a single type, then $f_n(v)$ is simply
$f_n$ evaluated at the type of $v$. 

\begin{lemma}\label{lfnk}
With the definitions above, if the functions $f_n$ are uniformly integrable, then
for each fixed $k$ we have
\[
 \left| \frac{1}{n} \sum_{v\,:\,|C_v|=k} f_n(v) - \int_\sss f_n(x)\rho_k(\ka;x) \dd\mu(x)  \right| \pto 0,
\]
where $C_v$ is the component of $G_n=G(A_n)$ containing the vertex $v$.
\end{lemma}
(For reader who prefers to define the cut metric via couplings, the corresponding formulation
of this lemma concerns functions $f_n$ defined on the spaces on which the kernels
$\ka$ and $\ka_{A_n}$ are coupled.)
\begin{proof}
We claim that it suffices to consider the case
where the $f_n$ are uniformly bounded.
Indeed, given any $\eps>0$, we may find uniformly bounded approximations $f_n'$
to $f_n$ with $\on{f_n'-f_n}\le \eps$ for every $n$. Applying the uniformly
bounded case, and then letting $\eps\to 0$, the result follows.
Rescaling, we may and shall assume that $\sn{f_n}\le 1$ for all $n$.

Using \refL{l_FWp} in place of \cite[Theorem 2.3]{cutsub},
the proof is now essentially the same as that of \cite[Lemma 2.11]{cutsub}, {\em mutatis
mutandis}. We only outline the changes.
Let
\[
 X_n = X_n(G_n) \= \frac{1}{n} \sum_{v\,:\,|C_v|=k} f_n(v).
\]
Adding or deleting an edge of $G_n$ changes $X_n$ by at most
$2k\sn{f_n}/n\le 2k/n$.
It follows that, arguing as in the proof of~\cite[Lemma 2.8]{cutsub}, we may
assume that the matrices $A_n$ are {\em well behaved}, meaning that all diagonal
entries are zero, and the maximum entry of $A_n$ is $o(n)$ as $n\to\infty$.
As in~\cite{cutsub}, we may then switch to the Poisson multigraph version of $G(A_n)$; we
omit the details. Using \cite[Lemma 2.10]{cutsub}, the contribution to $X_n$ from components
$C_v$ that contain cycles is then $\op(1)$.
On the other hand, the contribution from components isomorphic to some particular tree $T$
has expectation 
\[
 o(1) + (1+o(1)) \frac{\txp(T,f_n,\ka_n')}{\aut(T)};
\]
the argument is as for the corresponding relation (40) in~\cite{cutsub}. Continuing as in~\cite{cutsub},
but using \refL{l_FWp}, it follows that $|\E X_n - a_n|\to 0$,
where $a_n=\int_\sss f_n(x)\rho_k(\ka;x) \dd\mu(x)$. Considering sums over pairs of disjoint 
components, one obtains $|\E X_n^2-a_n^2|\to 0$, giving $|X_n - a_n|\pto 0$ as claimed.
\end{proof}

The corresponding result for the giant component is an immediate consequence; 
this is the natural analogue of~\cite[Theorem 9.10]{kernels} in the
present context.  

\begin{theorem}\label{tlf}
Let $(A_n)$ be a (deterministic or random)
sequence of matrices and $\ka$ an irreducible kernel on $\oi$
with $\dcut(A_n,\ka)\pto 0$, let $\ka_n'=\ka_{A_n}^{(\tau_n)}$ be a
(random) rearrangement of 
$\ka_{A_n}$ chosen so that $\cn{\ka_n'-\ka}\pto 0$,
and let $f_n$ be a uniformly integrable sequence of functions $f_n:\oi\to\RR$.
Then
\[
 \left| \frac{1}{n} \sum_{v\in \cC_1(G_n)} f_n(v) - \int_\sss f_n(x)\rho(\ka;x) \dd\mu(x)  \right| \pto 0,
\]
where $G_n=G(A_n)$, and $f_n(v)$ is defined by \eqref{Sv} and \eqref{fnv}.
\end{theorem}
\begin{proof}
As usual, by conditioning on the sequence $(A_n)$ (and now also on the
$\tau_n$), we may
assume that the $A_n$ are deterministic and $\cn{\ka_n'-\ka}\to 0$.

\refL{lfnk} extends immediately to a corresponding result summing over all components of size $1\le k\le K$
for any fixed $K$, and hence for $K=K(n)\to\infty$ sufficiently slowly. But the results of~\cite{cutsub}
show that only $\op(n)$ vertices in components of size more than
$K(n)$ are not in $\cC_1$, and conversely, trivially, at most $K(n)=o(n)$
vertices of $\cC_1$ are not in such components, so we obtain 
\[
 \left| \frac{1}{n} \sum_{v\notin \cC_1(G_n)} f_n(v) - \int_\sss f_n(x) (1-\rho(\ka;x)) \dd\mu(x)  \right| \pto 0.
\]
It remains only to note that
\[
 \frac{1}{n} \sum_{v\in V(G_n)} f_n(v) = \int_\sss f_n(x) \dd\mu(x)
\]
by definition of $f_n(v)$.
\end{proof}

We shall need the following simple observation concerning the cut norm. In this we
write $\cnm{\ka}{\mu}$ for the cut norm of $\ka$ defined with respect to a measure
$\mu$.
\begin{lemma}\label{lmult}
Let $\ka$ be a kernel on a measure space $(\sss,\mu)$ with $0<\mu(\sss)<\infty$,
and let $h$ be a non-negative measurable function on $(\sss,\mu)$.
Let $\nu$ be the measure defined by $\dd\nu(x)=h(x)\dd\mu(x)$.
Then
\[
 \cnm{\ka}{\nu} \le \sn{h}^2\cnm{\ka}{\mu}.
\]
\end{lemma}
\begin{proof}
Essentially immediate from \eqref{cntdef}. Indeed, for any $f$, $g$ with $\sn{f}$, $\sn{g}\le 1$,
\begin{multline*}
  \Bigabs{\int_{\sss^2} f(x)\ka(x,y)g(y)\dd\nu(x)\dd\nu(y)} \\
 =    \Bigabs{\int_{\sss^2} f(x)h(x)\ka(x,y)g(y)h(y)\dd\mu(x)\dd\mu(y)} \\
 =   \sn{h}^2  \Bigabs{\int_{\sss^2} \tf(x)\ka(x,y)\tg(y)\dd\mu(x)\dd\mu(y)},
\end{multline*}
where $\tf(x)=f(x)h(x)/\sn{h}$ has $\sn{\tf}\le 1$, and similarly for $\tg$.
The final integral is bounded by $\cnm{\ka}{\mu}$ by definition.
\end{proof}

Using \refT{tlf}, it is not hard to prove \refT{th1}.
\begin{proof}[Proof of \refT{th1}]
Given a kernel $\ka$ and real number $\delta$, let
\[
 m_{\delta}(\ka) = \sup_{\mu(A)\le \delta} \int_{A\times\sss}\ka(x,y) \dd\mu(x)\dd\mu(y),
\]
so $m_{\delta}(\ka)$ is the integral of the marginal of $\ka$ over the set
with measure $\delta$ where this marginal is maximal.
Note that
\[
 | m_\delta(\ka_1)-m_\delta(\ka_2) | \le \dcut(\ka_1,\ka_2).
\]
Also, if $\ka$ is integrable, then $m_{\delta}(\ka)\to 0$ as $\delta\to 0$.

Suppose now that $\ka$ is irreducible, and,
conditioning as usual, that the $A_n$ are deterministic with $\dcut(A_n,\ka)\to 0$. 
Suppose also that $\rho(\ka)>0$; otherwise, $\hhka=\ka$, while 
by \eqref{gc} the matrices $\tA_n$ are obtained from $A_n$
by deleting $\op(n)$ rows and columns, and the result follows easily.

Let $\ka_n'=\ka_n^{(\tau_n)}$ be a rearrangement of $\ka_{A_n}$ chosen so that
$\cn{\ka_n'-\ka}\to 0$.
As before, let  $S_v=S_{v,n}=\taun\qw((v-1)/n,v/n])$ be the subset
of $\sss$ corresponding to a vertex $v$ of $G_n$ under the rearrangement $\tau_n$.

Let $\nu_n$ be the {\em random} 
measure that agrees with Lebesgue measure $\mu$ on each $S_v$, $v\notin \cC_1$,
and is zero otherwise. Noting that $\nu_n(\sss)=(1-|\cC_1|/n)\pto \hmu(\sss)=1-\rho(\ka)$,
let $\nu_n'=\nu_n/\nu_n(\sss)$ be the rescaled version of $\nu_n$. 

Although it may appear that we have done our best to disguise this
fact, the kernel  
$\ka_n'$ on the measure space $(\sss,\nu_n')$ is simply a rearrangement of the kernel 
$\ka_{\tA_n}$, where $\tA_n$ is the submatrix of $A_n$ obtained by deleting rows and columns
corresponding to vertices in $\cC_1(G_n)$. Since $\dcut$ is unchanged by rearrangement,
indicating now the measure on the space (always $\sss$) on which our kernels are defined,
our aim is exactly to show that
\begin{equation}\label{aim}
 \dcut( (\ka_n',\nu_n'), (\ka,\hmu'))\pto 0.
\end{equation}
 
Fix $\eps>0$. From the comments at the start of the proof there is some $\delta$
such that $m_{\delta}(\ka)<\eps/2$, and then
\begin{equation}\label{mds}
 m_{\delta}(\ka_n')<\eps
\end{equation}
for $n$ large enough.

Let $\ka_f$ be a finite-type kernel approximating $\ka$ within $\eps$ in the $L^1$ norm, and hence
in $\dcut$:
\begin{equation}\label{kfk}
 \cn{\ka_f-\ka} \le \on{\ka_f-\ka} \le \eps,
\end{equation}
with $\ka_f$ constant on
the sets $A_i\times A_j$ for some partition $A_1,\ldots,A_r$ of $\sss$
into measurable sets.

Fix (for the moment) $1\le i\le r$. Applying \refT{tlf} with every $f_n$ equal to the indicator
function of $A_i$, we see that
\[
 \sum_{v\in \cC_1} \mu(S_v\cap A_i) \pto \int_{A_i}\rho(\ka;x)\dd\mu(x).
\]
Let 
\[
 \nu_{n,i} =  \sum_{v\notin \cC_1} \mu(S_v\cap A_i) = \mu(A_i)- \sum_{v\in \cC_1} \mu(S_v\cap A_i).
\]
Then, recalling \eqref{hmudef}, 
we have
\[
 \nu_{n,i}\pto \int_{A_i} (1-\rho(\ka;x))\dd\mu(x) = \hmu(A_i)
\]
for each $i$, and hence
\[
 \Delta \= \sum_{i=1}^r |\nu_{n,i}-\hmu(A_i)| \pto 0.
\]

Since our aim is to prove an `in probability' result, coupling appropriately,
we may condition on the random graphs $G_n$, and assume that $\Delta\to 0$. 
Note that all quantities we consider are now deterministic.

Recall that $\nu_n'=\nu_n/\nu_n(\sss)$ is the rescaled version of $\nu_n$,
and $\hmu'=\hmu/(1-\rho(\ka))=\hmu/\hmu(\sss)$ is the rescaled version
of $\hmu$.  
Since $\nu_n(A_i)=\nu_{n,i}$, it follows that
\[
 \Delta' \= \sum_{i=1}^r |\nu_n'(A_i)-\hmu'(A_i)| \to 0.
\]

Let $\nu_n^*$ be obtained by `tweaking' $\nu_n'$ so that $\nu_n^*(A_i)=\hmu'(A_i)$ for every $i$.
More precisely, recalling that the unnormalized measure $\nu_n$ had a $0/1$-valued
density function $f(x)=\dd\nu_n/\dd\mu$, we change $f$ on a set of measure $\Delta$
to obtain a $0/1$-valued $f'$ with $\int_{A_i}f'(x)\dd\mu(x)=\hmu'(A_i)$, and use $f'$ to define
the (normalized) measure $\nu_n^*$. Since the normalizing factors
are bounded (in the limit), it is not hard to check that for some constant $C$ we have 
\[
 \dcut( (\ka_n',\nu_n'), (\ka_n',\nu_n^*) ) \le 2m_{C\Delta}(\ka_n')
\]
for all large enough $n$. Indeed, we may couple the measures $\nu_n'$ and $\nu_n^*$
to agree with probability at least $1-C\Delta$. Alternatively,
we may rearrange the kernels to differ only where one
or both coordinates fall into some set of measure at most $C\Delta$.
(If we had $\int f=\int f'$, we could take this set to be simply the
set where $f$ and $f'$ differ.)

Since $\Delta\to 0$, \eqref{mds} shows that the right hand side above is $O(\eps)$.
Using $\ka_f$, it is now easy to complete the proof.
Note that $\cn{\ka_n'-\ka_f}\le \cn{\ka_n'-\ka}+\cn{\ka-\ka_f} \le \eps+o(1)$,
from \eqref{kfk} and our convergence assumption.
By Lemma~\ref{lmult}, we thus have
\[
 \dcut( (\ka_n',\nu_n^*) , (\ka_f,\nu_n^*) ) \le O(\eps)+o(1).
\]
Since $\ka_f$ is constant on each set $A_i\times A_j$, the kernel $\ka_f$ `only cares how much
measure falls in each $A_i$', and we have
\[
 \dcut( (\ka_f,\nu_n^*), (\ka_f,\hmu') )=0.
\]
But by \refL{lmult} again,
\[
 \dcut( (\ka_f,\hmu'), (\ka,\hmu') ) \le \cnm{\ka_f-\ka}{\hmu'} \le \cn{\ka_f-\ka} \le \eps.
\]
Putting the last four displayed inequalities together and using the triangle inequality, \eqref{aim} follows.
\end{proof}

\begin{ack}
This research was carried out during a visit of both authors 
to the programme ``Discrete Probability''
at Institut Mittag-Leffler, Djurs\-holm, Sweden, 2009.
\end{ack}

\newcommand\RSA{\emph{Random Struct. Alg.} }

\newcommand\vol{\textbf}
\newcommand\jour{\emph}
\newcommand\book{\emph}
\newcommand\inbook{\emph}
\def\no#1#2,{\unskip#2, no. #1,} 
\newcommand\toappear{\unskip, to appear}

\newcommand\webcite[1]{
\texttt{\def~{{\tiny$\sim$}}#1}\hfill\hfill}
\newcommand\webcitesvante{\webcite{http://www.math.uu.se/~svante/papers/}}
\newcommand\arxiv[1]{\webcite{arXiv:#1.}}

\end{document}